\documentclass[12pt]{amsart}

\usepackage{amsmath, amssymb}
\usepackage{hyperref}
\DeclareMathOperator{\supp}{supp}

\newtheorem{theorem}{Theorem}[section]
\newtheorem{lemma}[theorem]{Lemma}

\theoremstyle{definition}
\newtheorem{definition}[theorem]{Definition}

\theoremstyle{remark}
\newtheorem{remark}[theorem]{Remark}

\numberwithin{equation}{section}
\usepackage{bbm}

\begin{document}

\title[Full-dimensional projections of measures]{Full packing dimensional projections of measures}


\author{Nicolas Angelini}

\curraddr{}
\address{Instituto de Matemática Aplicada San Luis (IMASL), CONICET, Argentina.  
Departamento de Matemática, FCFMyN, Universidad Nacional de San Luis, Argentina.}
\email{nicolas.angelini.2015@gmail.com}
\thanks{}


\date{\today}

\begin{abstract}
We introduce a threshold parameter $D(\mu)$
for a Borel probability measure $\mu$ with compact support
$E\subset\mathbb{R}^n$ such that, for
every integer $1\leq m\leq n$, the orthogonal projection of $\mu$ onto a
typical $m$-dimensional subspace attains full packing dimension if and
only if $m\leq D(\mu)$. In the complementary regime we show that the Assouad dimension of the
support controls the possible drop of the packing dimension under
projections:
$$\dim_P^{m}\mu\geq\dim_P\mu-\max\{0,\ \dim_A E-m\}.$$
In particular, whenever $m\geq\dim_A E$, the packing dimension of every
measure supported on $E$ is preserved under orthogonal projection onto
almost every $m$-dimensional subspace. Taking supremum over the measures
supported on a set recovers, in its Assouad dimension form, the
corresponding result of Falconer, Fraser and Shmerkin for sets. A key
ingredient, of independent interest, is a sharpening of an estimate of
Falconer and Mattila for the growth of the measure of balls, in which the
ambient dimension is replaced by the Assouad dimension of the support.
\end{abstract}


\maketitle

\section{Introduction}

Let $\mu$ be a finite Borel measure on $\mathbb{R}^n$, and for $V\in G(n,m)$ let $\pi_V$ denote the orthogonal projection onto $V$. We write $\mu_V=\mu\circ \pi_V^{-1}$ for the projected measure.

For Hausdorff dimension, it is classical that
\begin{equation}\label{ProjHausIntro}
\dim_H \mu_V=\min\{\dim_H\mu,m\}
\end{equation}
for $\gamma_{n,m}$-almost every $V\in G(n,m)$, see \cite{hu1994fractal}. In contrast, the behavior of packing dimension under projections is more subtle.

Falconer and Howroyd introduced packing dimension profiles, $\dim_P^m$, and proved that for any finite Borel measure $\mu$ on $\mathbb{R}^n$,
\[
\dim_P \mu_V\le \dim_P^m \mu
\]
for all $V\in G(n,m)$, with equality for $\gamma_{n,m}$-almost every $V$, \cite{falconer1997packing}. This shows that the packing dimension of typical projections is almost surely constant and determined by the $m$-dimensional profile of $\mu$. However, this result does not directly yield analogues of \eqref{ProjHausIntro}.

It is straightforward to see that
\[
\dim_P \mu_V \le \min\{\dim_P\mu,m\}
\]
for all $V\in G(n,m)$. On the other hand, Falconer and Mattila obtained a general lower bound: if $\mu$ is a Borel probability measure on $\mathbb{R}^n$, then for $\gamma_{n,m}$-almost all $V$,
\[
\dim_P \mu_V \ge 
\begin{cases}
\dfrac{\dim_P \mu\left(1-\frac{1}{n}\dim_H \mu\right)}
{1+\left(\frac{1}{m}-\frac{1}{n}\right)\dim_P \mu-\frac{1}{m}\dim_H \mu},
& \text{if } \dim_H \mu \le m, \\[1.2em]
m, & \text{if } \dim_H \mu \ge m,
\end{cases}
\]
and they showed that this bound is sharp in general \cite{falconer1996packing}.

The main purpose of this article is to identify conditions under which Hausdorff-type projection formulas can be recovered for packing dimension. We introduce a parameter $D(\mu)$ (Definition~\ref{maindefinition}) that governs when projections have full dimension, in the sense that
\[
\dim_P \mu_V = m \quad \text{for } \gamma_{n,m}\text{-a.e. }V
\quad \Longleftrightarrow \quad
m\le D(\mu),
\]
see Theorem~\ref{maintheorem}. This provides a complete characterization of the regime in which typical projections attain maximal dimension.

To obtain lower bounds in the complementary case, $D(\mu)<m$, we relate
the packing dimension profiles to the Assouad dimension of the support of
$\mu$. Our main estimate (Theorem~\ref{thm:profile_assouad_bound})
asserts that
\[
\dim_P^m \mu \ge \dim_P \mu - \max\{0,\dim_A(\supp\mu)-m\};
\]
note that the measure enters this bound only through $\dim_P\mu$, so that
the size of the possible loss is governed by the support alone and is
common to all measures supported on a given set. For integer $m$, by the
projection theorem of Falconer and Howroyd this reads
\[
\dim_P \mu_V \ge \dim_P \mu - \max\{0,\dim_A(\supp\mu)-m\}
\]
for $\gamma_{n,m}$-almost every $V\in G(n,m)$. In particular, if
$m\ge \dim_A(\supp\mu)$, then $\dim_P \mu_V=\dim_P\mu$ for typical
projections: the Assouad dimension of the support controls the stability
of the packing dimension under orthogonal projections.

The corresponding statement for sets was obtained by Falconer, Fraser and
Shmerkin \cite{falconer2021assouad}. Their argument is not available here, it relies on the
capacity formulation of the box dimension profiles, in which one is free
to choose a different measure at each scale, whereas in our setting the
measure is given and is the same at every scale.
Theorem~\ref{thm:profile_assouad_bound} applies to each individual measure, and taking
supremum over the measures supported on a set recovers their result in its
Assouad dimension form (see Remark~\ref{rem:not_adaptation}). Analogous
results for the $\theta$-intermediate dimensions were obtained in
\cite{AngeliniMolter2025}, see also \cite{ANGELINI2026130039}.

\section{Preliminaries}

We write $B(x,r)$ to the ball of center $x$ and radious $r$ and for a bounded set $E\subset \mathbb{R}^n$ we write $|E|$ for its diameter.

Throughout the paper, $\mu$ denotes a Borel probability measure on $\mathbb{R}^n$, i.e. $\mu(\mathbb{R}^n)=1$. We write $\supp \mu$ for the support of $\mu$.

For $E\subset\mathbb{R}^n$ Borel, $\mathcal{M}_c^+(E)$ denotes the set of
Borel probability measures with compact support contained in $E$.

For $1\le m\le n$, we write $G(n,m)$ for the Grassmannian of $m$-dimensional linear subspaces of $\mathbb{R}^n$,
endowed with its natural invariant probability measure $\gamma_{n,m}$, see \cite[Chapter 3]{mattila1999geometry}.

For $V\in G(n,m)$, let $\pi_V:\mathbb{R}^n\to V$ denote the orthogonal projection onto $V$, and define the projected measure by
\[
\mu_V=\mu\circ \pi_V^{-1}.
\]
This is a finite Borel measure on $V$.

Hausdorff dimension will be denoted by $\dim_H$. 

The Assouad dimension of a set $E\subset\mathbb{R}^n$, $\dim_A E$ is defined by

\begin{align*}
\dim_A E := \inf \Bigl\{ s \ge 0 :\ &\exists C>0 \text{ such that for all } 0<r<R,\ x\in E,\\
&N_r\bigl(E \cap B(x,R)\bigr)
\le C \left(\frac{R}{r}\right)^s
\Bigr\}.
\end{align*}

where $N_r(A)$ denotes the minimal number of balls of radius $r$ needed
to cover $A$. The value of $\dim_A E$ is unchanged if the covering balls
are replaced by cubes of side length $r$, by the cubes of the mesh of
side length $r$ oriented at the origin that intersect $A$, or by maximal
$r$-separated subsets of $A$, and also if the ball $B(x,R)$ is replaced
by the cube $Q(x,R)$ of side length $R$ centered at $x$. We use these alternative formulations interchangeably, and
refer to \cite{fraser2020assouad} for a detailed discussion.

Throughout, we shall use without further mention that, if $\mu$ is a
Borel probability measure on $\mathbb{R}^n$ with compact support
$E$, then
$$\dim_P\mu\leq\dim_P E\leq\dim_A E\leq n.$$

Let $0<s<\infty$, a closed set $E\subset\mathbb{R}^n$ is called \emph{$s$-regular} if there exists Borel measure $\mu$ on $\mathbb{R}^n$ and a constant $C\ge 1$ such that $\mu(\mathbb{R}^n\setminus E)=0$ and
\[
C^{-1} r^s \le \mu\bigl(E\cap B(x,r)\bigr)
\le C r^s
\]
for all $x\in E$ and all $0<r<|E|$, $r<\infty$.

The packing dimension of a Borel measure $\mu$ is defined by

\[
\dim_P \mu = \inf\{ \dim_P E : E \text{ is a Borel set with }\mu(E) > 0 \},
\]

where $\dim_P E$ denote the packing dimension of the set $E$.

An equivalent characterization is as follows, see \cite[Lemma 4.1]{hu1994fractal}
\[
\dim_P \mu =\sup \left\{ t\geq 0\, :\, \liminf_{r \to 0} \frac{\mu(B(x,r))}{r^t} = 0 \, \text{ for }\mu-a.e.x\in\mathbb{R}^n\right\}.
\]

After that, in \cite{falconer1997packing}, Falconer and Howroyd gives a characterization of packing dimension in terms of the potential defined by 
$$F_n^\mu (x,r)=\int_{\mathbb{R}^n} \min \{1,r^n|x-y|^{-n}\}d\mu(y).$$

\begin{lemma}\cite[Corollary 3]{falconer1997packing}
    For $\mu$ a finite Borel measure on $\mathbb{R}^n$ we have
    $$\dim_P \mu = \sup\left\{ t\geq 0\, :\, \liminf_{r\to 0}\frac{F_n^{\mu}(x,r)}{r^t}=0 \text{ for }\mu-a.e.x\in\mathbb{R}^n\right\}.$$
\end{lemma}

In the same article, in order to obtain sharp projection theorems, the authors introduced packing dimension profiles. For each $0 \leq m \leq n$, the $m$-dimensional packing dimension profile of a measure $\mu$ is defined as follows.

\begin{definition}
    For a finite Borel measure $\mu$ on $\mathbb{R}^n$ and $0\leq m\leq n$ define
    $$F_m^\mu (x,r)=\int_{\mathbb{R}^n} \min \{1,r^m|x-y|^{-m}\}d\mu(y)$$
    and 
   $$\dim_P^m \mu = \sup\left\{ t\geq 0\, :\, \liminf_{r\to 0}\frac{F_m^{\mu}(x,r)}{r^t}=0 \text{ for }\mu-a.e.x\in\mathbb{R}^n\right\}.$$
\end{definition}

The fundamental result of Falconer and Howroyd asserts that for a finite Borel measure $\mu$ on $\mathbb{R}^n$, the packing dimension of the projection $\mu_V$ onto a typical $m$-dimensional subspace $V$ is given by the $m$-dimensional packing dimension profile of $\mu$:

\begin{theorem}\label{ProjPacking}\cite[Theorem 6]{falconer1997packing}
    Let $\mu$ be a finite Borel measure on $\mathbb{R}^n$. Then, for all $V\in G(n,m)$ we have
    $$\dim_P \mu_V \leq \dim_P^m \mu $$
    with an equality for $\gamma_{n,m}-a.e.V\in G(n,m)$.
\end{theorem}

The best lower bound was obtained in \cite{falconer1996packing} by Mattila and Falconer

\begin{theorem}\cite[Theorem 3.3]{falconer1996packing}
    Let $\mu$ be a Borel probability measure on $\mathbb{R}^n$. Then for $\gamma_{n,m}$ almost all $V\in G(n,m)$, 
    \[
\dim_P \mu_V \ge 
\begin{cases}
\dfrac{\dim_P \mu\left(1-\frac{1}{n}\dim_H \mu\right)}
{1+\left(\frac{1}{m}-\frac{1}{n}\right)\dim_P \mu-\frac{1}{m}\dim_H \mu},
& \text{if } \dim_H \mu \le m, \\[1.2em]
m, & \text{if } \dim_H \mu \ge m.
\end{cases}
\]
\end{theorem}

Finally, it is well known (see \cite{hu1994fractal}, or
\cite[(4.1)]{falconer1997packing}) that
\begin{equation}\label{eq:dimP_set_sup}
\dim_P E=\sup\{\dim_P\mu\ :\ \mu\in\mathcal{M}_c^+(E)\},
\end{equation}
which led Falconer and Howroyd to define the $s$-packing dimension
profile of $E$, for $0\leq s\leq n$, by
$$\dim_P^s E:=\sup\{\dim_P^s\mu\ :\ \mu\in\mathcal{M}_c^+(E)\},$$
see \cite[(5.3)]{falconer1997packing}.

Falconer, Fraser and Shmerkin proved that the Assouad dimension of a set
controls the possible drop of its packing dimension profiles. To state
their result, recall the upper Assouad spectrum
$\overline{\dim}_A^{\theta}$, $\theta\in(0,1)$, defined exactly as the
Assouad dimension but with the scales restricted to
$0<r\leq R^{1/\theta}<R<1$; it satisfies
$\overline{\dim}_A^{\theta}F\leq\dim_A F$, see \cite{fraser2020assouad}.

\begin{theorem}[{\cite[Theorem 2.2]{falconer2021assouad}}]\label{thm:FFS}
Let $s\in(0,n]$ and $\theta\in(0,1)$. If $F\subseteq\mathbb{R}^{n}$ is a
Borel set, then
$$\dim_P^{s}F\ \geq\ \dim_P F-\max\bigl\{0,\ \overline{\dim}_A^{\theta}F-s,\
(\dim_A F-s)(1-\theta)\bigr\}.$$
In particular, since the maximum above never exceeds
$\max\{0,\dim_A F-s\}$,
\begin{equation}\label{eq:FFS_assouad_form}
\dim_P^{s}F\ \geq\ \dim_P F-\max\{0,\ \dim_A F-s\},
\end{equation}
which we refer to as the \emph{Assouad dimension form} of the theorem.
\end{theorem}

In \cite{falconer2021assouad} the profiles $\dim_P^{s}F$ are defined
through capacities; for integer exponents this definition coincides with
the one given above, since both quantities equal the packing dimension
of $\pi_V F$ for $\gamma_{n,m}$-almost every $V\in G(n,m)$, see
\cite[Theorem 10]{falconer1997packing} and
\cite[Theorem 2.1]{falconer2021assouad}.

\section{Estimates for the measure of balls across scales}
The goal of this section is Lemma~\ref{Lemma 3.9}, an upper bound
for the measure of balls $\mu(B(x,r))$ valid, at $\mu$-almost every
point, along suitable windows of scales of the form
$[\delta^{a},\delta^{a\theta}]$. The exponent in the bound involves the
packing dimension of $\mu$ and the Assouad dimension of its support,
rather than the ambient dimension $n$. To achieve this, we first show
that every compact set of Assouad dimension less than $n$ is contained
in a compact $s-$regular set of exponent $s$ arbitrarily close to its
Assouad dimension (Lemma~\ref{Assouad Lemma}); running the argument of
Falconer and Mattila \cite{falconer1996packing} against the resulting
regular measure, instead of the Lebesgue measure, then yields the
required growth estimate (Lemma~\ref{lem:assouad_measure_growth}).

\begin{lemma}\label{Assouad Lemma}
Let $E \subset \mathbb{R}^n$ be a compact set with $\dim_{\mathrm A} E < n$.
Then, for every $\varepsilon>0$, there exist
$t\in(\dim_A E,\, \dim_A E+\varepsilon)$, a compact set
$F\subset\mathbb{R}^n$ with $E\subset F$, and a $t$-regular Borel
probability measure $\mu$ with $\operatorname{supp}\mu=F$.
\end{lemma}

\begin{proof}
Without loss of generality we may assume $E \subset [0,1)^n$. Let $s=\dim_A E\text{ and } \epsilon>0$ sufficiently small such that
$t:=s+\varepsilon/2<n$. For $m\in\mathbb{N}_0$, let $\mathcal{D}_m$ denote
the family of half-open dyadic subcubes of $[0,1)^n$ of side length
$2^{-m}$, and for a bounded set $A$ let $N_r(A)$ denote the number of
cubes in the mesh of side length $r$, oriented at the origin, that
intersect $A$. By the definition of the Assouad dimension, there exists
$C\geq 1$ such that for all $0<r<R$ and all $x\in E$,
\begin{equation}\label{eq:assouad_bound}
N_r\bigl(E \cap Q(x,R)\bigr)
\le C \left(\frac{R}{r}\right)^{t},
\end{equation}
where $Q(x,R)$ denotes the cube of side length $R$ centered at $x$.

Fix a positive integer $g$ large enough such that
\begin{equation}\label{def:gap}
    M := \left\lceil C\,2^{t}\,2^{gt} \right\rceil < 2^{gn}
    \qquad\text{and}\qquad
    \frac{t+1+\log_2 C}{g}\leq \frac{\varepsilon}{2}.
\end{equation}
Define $t':=\frac{\log_2 M}{g}.$ Since $C\geq 1$, we have $M\geq 2^{gt}$ and
$M\leq C\,2^t\,2^{gt}+1\leq 2C\,2^t\,2^{gt}$, whence
$$t\leq t'\leq t+\frac{t+1+\log_2 C}{g}\leq t+\frac{\varepsilon}{2}
=s+\varepsilon.$$
Set $m_k:=kg$ and $r_k:=2^{-m_k}$ for $k\geq 0$, and note that
\begin{equation}\label{eq:exact_scaling}
    M^{-k}=r_k^{\,t'}\qquad\text{for every }k\geq 0.
\end{equation}

\emph{Construction of $F$.} Let $\mathcal{F}_0:=[0,1)^n$. Suppose
that for some $k\geq 0$ we have constructed a family of dyadic cubes
$\mathcal{F}_k\subset\mathcal{D}_{m_k}$ such that
$E\subset F_k:=\bigcup_{Q\in\mathcal{F}_k}Q$. For each
$Q\in\mathcal{F}_k$, partition $Q$ into its $2^{gn}$ dyadic children in
$\mathcal{D}_{m_{k+1}}$, retain all children $Q'$ with
$E\cap Q'\neq\emptyset$, and then add further children of $Q$, chosen
arbitrarily, until exactly $M$ children of $Q$ have been selected. This
is always possible: on the one hand, if $Q\cap E\neq\emptyset$, pick
$x\in Q\cap E$; since $Q\subset Q(x,2r_k)$, inequality
\eqref{eq:assouad_bound} gives
\begin{equation}\label{ineq:children}
    N_{r_{k+1}}(E\cap Q)
    \leq N_{r_{k+1}}\bigl(E\cap Q(x,2r_k)\bigr)
    \leq C\,2^t\left(\frac{r_k}{r_{k+1}}\right)^{t}\leq M,
\end{equation}
so all children that intersects $E$ can be retained; on the other hand, by
\eqref{def:gap} we have $M<2^{gn}$, so there are always enough children
available to complete the selection. Let $\mathcal{F}_{k+1}$ be the
family of all selected children of all cubes of $\mathcal{F}_k$, and
$F_{k+1}:=\bigcup_{Q\in\mathcal{F}_{k+1}}Q$. Since every cube of
$\mathcal{D}_{m_{k+1}}$ meeting $E$ is a child of some cube of
$\mathcal{F}_k$ (because $E\subset F_k$), we get $E\subset F_{k+1}$, and
by construction $F_{k+1}\subset F_k$ and each cube of $\mathcal{F}_k$
contains exactly $M$ cubes of $\mathcal{F}_{k+1}$; in particular
$\#\mathcal{F}_k=M^k$. Define
\[
F := \bigcap_{k=0}^{\infty} \overline{F_k},
\]
a compact set containing $E$.

Now, Since every cube of $\mathcal{F}_k$ has exactly $M$
selected children in $\mathcal{F}_{k+1}$, setting
\[
\mu(Q):=M^{-k}\qquad \text{for every }Q\in\mathcal{F}_k
\]
defines a consistent mass distribution, which by \cite[Proposition 1.7]{falconer2013fractal} extends to a Borel
probability measure $\mu$ with $\operatorname{supp}\mu\subset F$.

Now, if we  choose $k\geq 0$ such that
$r_{k+1}\leq r<r_k$, since $B(x,r)\subset Q(x,2r_k)$ and a cube of side
length $2r_k$ meets at most $3^n$ cubes of the mesh $\mathcal{D}_{m_k}$,
each of which has $\mu$-mass at most $M^{-k}$, we obtain from
\eqref{eq:exact_scaling} and $r_k=2^g r_{k+1}\leq 2^g r$ that
$$\mu(B(x,r))\leq 3^n M^{-k}=3^n r_k^{\,t'}\leq 3^n\,2^{gn}\, r^{\,t'}.$$

For the lower bound, let $x\in F$ and $0<r\leq 1$, and choose $k\geq 0$
with $r_{k+1}\leq \frac{r}{2\sqrt n}<r_k$. Since
$x\in\overline{F_{k+1}}$, which is a finite union of closed cubes, there
exists $Q'\in\mathcal{F}_{k+1}$ with $x\in\overline{Q'}$; as $Q'$ has
diameter $\sqrt n\, r_{k+1}$, we get
$Q'\subset B(x,2\sqrt n\, r_{k+1})\subset B(x,r)$. Hence, using
\eqref{eq:exact_scaling}, $r_{k+1}>2^{-g}\frac{r}{2\sqrt n}$ and
$t'<n$,
$$\mu(B(x,r))\geq \mu(Q')=r_{k+1}^{\,t'}
\geq \left(\frac{2^{-g}}{2\sqrt n}\right)^{n} r^{\,t'}.$$
Taking $c:=\max\{3^n 2^{gn},\,(2^{g+1}\sqrt n)^{n}\}$ completes the
proof.
\end{proof}

The following lemma sharpens an estimate of Falconer and Mattila
\cite{falconer1996packing}, replacing the ambient dimension $n$ in the
exponent by the Assouad dimension of the support. The key new ingredient
is the $t$-regular measure provided by Lemma~\ref{Assouad Lemma}, which
plays the role of the $n$-dimensional Lebesgue measure in their argument.

\begin{lemma}\label{lem:assouad_measure_growth}
Let $\mu$ be a Borel probability measure on $\mathbb{R}^n$ with compact
support $E$ satisfying $0<\dim_A E$. Let $a\in(0,1)$ and $\varepsilon>0$,
and set $s:=\dim_A E$. Then for $\mu$-almost every $x\in E$ there exists
$\rho_0(x)>0$ such that
\begin{equation}\label{eq:ball_growth}
\mu\bigl(B(x,r)\bigr)
\le \left(\frac{4r}{\rho}\right)^{s(1+\varepsilon)}\mu\bigl(B(x,\rho)\bigr),
\end{equation}
for all $0<\rho<\rho_0(x)$ and $\rho^{a}\le r\le 1.$\end{lemma}

\begin{proof}
If $s=n$ this is the result of \cite{falconer1996packing}, so assume
$s<n$. Fix auxiliary parameters $\delta>0$ and $\eta>0$, to be chosen at
the end so that $t(1+\eta)\leq s(1+\varepsilon)$, where
$t\in(s,s+\delta)$ is the exponent provided by
Lemma~\ref{Assouad Lemma}; let $\nu$ be the corresponding $t$-regular
probability measure, supported on a compact set $F\supset E$, with
regularity constant $c\geq 1$:
\begin{equation}\label{eq:nu_regular}
c^{-1}r^{t}\leq \nu(B(y,r))\leq c\,r^{t}
\qquad\text{for all }y\in F,\ 0<r\leq 1.
\end{equation}
Note two consequences of \eqref{eq:nu_regular}: the lower bound applies
at $\mu$-a.e.\ point, since $\operatorname{supp}\mu=E\subset F$; and the
upper bound extends to arbitrary centers $y\in\mathbb{R}^n$ at the cost
of a factor $2^t$, since if $B(y,r)\cap F\neq\emptyset$ then
$B(y,r)\subset B(y',2r)$ for some $y'\in F$, whence
$\nu(B(y,r))\leq c\,2^{t}r^{t}$ for all $y$ and $0<r\leq 1/2$.

For nonnegative integers $p<q$ define
$$A_{p,q}:=\bigl\{x\in E:\ \mu(B(x,2^{-p}))>2^{(q-p)t(1+\eta)}\,\mu(B(x,2^{-q}))\bigr\}.$$
We claim that
\begin{equation}\label{eq:dyadic_bound}
\mu(A_{p,q})\leq 8^{t}c^{2}\,2^{-(q-p)t\eta}.
\end{equation}
First observe that for every $x\in\mathbb{R}^n$,
\begin{equation}\label{eq:local_bound}
\mu\bigl(A_{p,q}\cap B(x,2^{-q-1})\bigr)\leq
2^{(p-q)t(1+\eta)}\,\mu\bigl(B(x,2^{-p+1})\bigr).
\end{equation}
Indeed, if the left-hand side is zero there is nothing to prove;
otherwise pick $y\in A_{p,q}\cap B(x,2^{-q-1})$. Any two points of
$B(x,2^{-q-1})$ lie within distance $2^{-q}$ of each other, so
$A_{p,q}\cap B(x,2^{-q-1})\subset B(y,2^{-q})$, and since $y\in A_{p,q}$
means $\mu(B(y,2^{-q}))<2^{(p-q)t(1+\eta)}\mu(B(y,2^{-p}))$, we get

\begin{align*}
    \mu\bigl(A_{p,q}\cap B(x,2^{-q-1})\bigr)&\leq\mu(B(y,2^{-q}))\\
    &<2^{(p-q)t(1+\eta)}\mu(B(y,2^{-p}))\\
    &\leq 2^{(p-q)t(1+\eta)}\mu(B(x,2^{-p+1})).
\end{align*}

Now we estimate $\mu(A_{p,q})$. By the lower bound in
\eqref{eq:nu_regular}, valid at $\mu$-a.e.\ $x$,
$$\mu(A_{p,q})\leq c\,2^{(q+1)t}\int_{A_{p,q}}\nu\bigl(B(x,2^{-q-1})\bigr)\,d\mu(x).$$
Writing $\nu(B(x,2^{-q-1}))=\int\mathbbm{1}_{\{|x-y|\leq 2^{-q-1}\}}\,d\nu(y)$
and applying Fubini's theorem,

\begin{align*}
\int_{A_{p,q}}\nu\bigl(B(x,2^{-q-1})\bigr)\,d\mu(x)
&=\int \mu\bigl(A_{p,q}\cap B(y,2^{-q-1})\bigr)\,d\nu(y)\\
&\leq 2^{(p-q)t(1+\eta)}\int\mu\bigl(B(y,2^{-p+1})\bigr)\,d\nu(y),    
\end{align*}

by \eqref{eq:local_bound}. Applying Fubini once more and then the upper bound in \eqref{eq:nu_regular},
$$\int\mu\bigl(B(y,2^{-p+1})\bigr)\,d\nu(y)
=\int\nu\bigl(B(x,2^{-p+1})\bigr)\,d\mu(x)
\leq c\,2^{t}\,2^{-(p-1)t}.$$
Combining,
$$\mu(A_{p,q})\leq c\,2^{(q+1)t}\, 2^{(p-q)t(1+\eta)}\, c\,4^{t}\,2^{-pt}
= c^{2}\,2^{t}4^{t}\,2^{-(q-p)t\eta}\leq 8^{t}c^{2}\,2^{-(q-p)t\eta},$$
which proves \eqref{eq:dyadic_bound}.

Now, for $q_0\in\mathbb{N}$ let
$$
M_{q_0}:=\mu\left\{\begin{array}{l}
x : \mu(B(x,2^{-p})) > (2^{q-p})^{t(1+\epsilon)}\,\mu(B(x,2^{-q})) \\[4pt]
\text{for some integers } p \text{ and } q \text{ with } 0\leq p\leq aq \text{ and } q\geq q_0
\end{array}
\right\}$$

By \eqref{eq:dyadic_bound},
$$M_{q_0}\leq 8^{t}c^{2}\sum_{q=q_0}^{\infty}\sum_{p=0}^{\lfloor aq\rfloor}
2^{-(q-p)t\eta}
\leq 8^{t}c^{2}\,\frac{2^{t\eta}}{2^{t\eta}-1}\sum_{q=q_0}^{\infty}2^{-q(1-a)t\eta}
=c_{1}\,2^{-q_{0}(1-a)t\eta},$$
where $c_{1}$ depends only on $a,\eta,t$ and the regularity constant $c$
of $\nu$; here we used $q-p\geq q-aq=(1-a)q$ and summed the geometric
series in $p$.

Finally, let $\rho_{0}:=2^{-q_{0}-1}$,
and suppose $0<\rho<\rho_{0}$, $\rho^{a}\leq r\leq 1$ and
$$\mu(B(x,r))>\left(\frac{4r}{\rho}\right)^{t(1+\eta)}\mu(B(x,\rho)).$$

Choose integers $p,q$ with $2^{-p-1}<r\leq 2^{-p}$ and
$2^{-q}\leq\rho<2^{-q+1}$. Then $2^{-p}\geq r\geq\rho^{a}\geq 2^{-aq}$
gives $0\leq p\leq aq$, and $2^{-q}\leq\rho<2^{-q_{0}-1}$ gives
$q>q_{0}$. Moreover $\frac{4r}{\rho}\geq\frac{4\cdot 2^{-p-1}}{2^{-q+1}}
=2^{q-p}$, so

\begin{align*}
    2^{(q-p)t(1+\eta)}\mu(B(x,2^{-q}))&\leq \left(\tfrac{4r}{\rho}\right)^{t(1+\eta)}\mu(B(x,\rho))\\
    &<\mu(B(x,r))\\
    &\leq \mu(B(x,2^{-p})).
\end{align*}

Hence the set of $x$ for which \eqref{eq:ball_growth} fails, with
$t(1+\eta)$ in place of $s(1+\varepsilon)$, for some
$0<\rho<\rho_{0}$ and $\rho^{a}\leq r\leq 1$, has $\mu$-measure at most
$M_{q_{0}}\leq c_{1}2^{-q_{0}(1-a)t\eta}$.

The exceptional sets above decrease as
$q_{0}\to\infty$ and their measures tend to $0$; hence their
intersection is $\mu$-null, which means that for $\mu$-a.e.\ $x$ there
exists $\rho_{0}(x)>0$ such that
$$\mu(B(x,r))\leq\left(\frac{4r}{\rho}\right)^{t(1+\eta)}\mu(B(x,\rho)),$$
for all $0<\rho<\rho_{0}(x),\ \rho^{a}\leq r\leq 1$.
Finally, given $\varepsilon>0$, choose $\delta,\eta>0$ small enough that
$t(1+\eta)\leq (s+\delta)(1+\eta)\leq s(1+\varepsilon)$ and  \eqref{eq:ball_growth} follows.
\end{proof}

Combining Lemma~\ref{lem:assouad_measure_growth} with the definition of packing dimension of a measure, we obtain an
upper bound for the measure of balls along suitable windows of scales.

\begin{lemma}\label{Lemma 3.9}
Let $\mu$ be a Borel probability measure with compact support $E \subset \mathbb{R}^n$. 
Let $\theta, a \in (0,1)$, $\varepsilon>0$ and $ p \in (0,\dim_P \mu)$. Then for all $c>0$ we have that for $\mu$-almost every $x \in \mathbb{R}^n$ there exists a sequence $\delta_k \to 0$ satisfying
\[
\mu(B(x,r)) \leq c\, r^{\dim_A E(1+\varepsilon)-\frac{\dim_A E(1+\varepsilon)-p}{a\theta}},
\]
whenever $r \in [\delta_k^a,\delta_k^{a\theta}]$.
\end{lemma}
\begin{proof}
Write $s:=\dim_A E$ and note that $0<p\leq\dim_P\mu\leq s$, so that
$p-s(1+\varepsilon)<0$.

Let $c>0$ and set $c':=4^{-s(1+\varepsilon)}c$. By the characterization
of the packing dimension of a measure (see
\cite[Lemma 4.1]{hu1994fractal}), for $\mu$-almost every
$x\in\mathbb{R}^n$ we have
$\liminf_{\delta\to 0}\delta^{-p}\mu(B(x,\delta))=0$; in particular,
there exists a sequence $\delta_k\to 0$ (depending on $x$) such that
\[
\mu\bigl(B(x,\delta_k)\bigr)\leq c'\,\delta_k^{\,p}
\qquad\text{for every }k.
\]
On the other hand, by Lemma~\ref{lem:assouad_measure_growth}, for
$\mu$-almost every $x$ there exists $\delta_0(x)>0$ such that for all
$\delta\in(0,\delta_0(x))$ and all $r\in[\delta^a,1]$,
\[
\left(\frac{\delta}{4r}\right)^{s(1+\varepsilon)}
\leq \frac{\mu(B(x,\delta))}{\mu(B(x,r))}.
\]
Fix $x$ in the intersection of both full-measure sets and discard the
finitely many $k$ with $\delta_k\geq\delta_0(x)$. Combining the two
estimates, for all $r\in[\delta_k^{\,a},1]$,
\[
\mu\bigl(B(x,r)\bigr)
\leq (4r)^{s(1+\varepsilon)}\,\delta_k^{-s(1+\varepsilon)}\,
\mu\bigl(B(x,\delta_k)\bigr)
\leq c\;\delta_k^{\,p-s(1+\varepsilon)}\,r^{\,s(1+\varepsilon)}.
\]
In particular, if $r\in[\delta_k^{\,a},\,\delta_k^{\,a\theta}]$, then
$r^{1/(a\theta)}\leq\delta_k$, and since $p-s(1+\varepsilon)<0$,
\[
\delta_k^{\,p-s(1+\varepsilon)}\leq
r^{-\frac{s(1+\varepsilon)-p}{a\theta}},
\]
whence $\mu(B(x,r))\leq c\,
r^{\,s(1+\varepsilon)-\frac{s(1+\varepsilon)-p}{a\theta}}$.
\end{proof}

\section{Packing intermediate dimensions}

In this section we introduce the packing intermediate dimensions and the
threshold parameter $D(\mu)$ of a Borel probability measure, and
establish the inequality needed for the main theorem.

\begin{definition}\label{def:intermediate_profile}
Let $\mu$ be a Borel probability measure with compact support on
$\mathbb{R}^n$ and let $\theta\in(0,1]$. For $x\in\mathbb{R}^n$,
$r>0$ and $0\leq t \leq s$, write
$$F_{t,s,\theta}^{\mu}(x,r)=\int\min\{1,\ r^t|x-y|^{-t},\
r^{\theta(s-t)+t}|x-y|^{-s}\}\,d\mu(y),$$
and define
$$\dim_{P,\theta}\mu = \sup\left\{t\in[0,n]\ :\ \liminf_{r\to 0}
r^{-t}F_{t,n,\theta}^{\mu}(x,r)=0\ \text{ for } \mu\text{-a.e.\ }
x\in\mathbb{R}^{n}\right\}.$$
\end{definition}

\begin{definition}\label{maindefinition}
Let $\mu$ be a Borel probability measure with compact support on
$\mathbb{R}^n$. The \textit{threshold parameter} of $\mu$ is
$$D(\mu) := \lim_{\theta \to 0}\, \dim_{P,\theta}\mu.$$
\end{definition}

\begin{remark}\label{rem:monotone_theta}
The function $\theta\mapsto\dim_{P,\theta}\mu$ is
nondecreasing on $(0,1]$. In particular, the limit in
Definition~\ref{maindefinition} exists and we have 
$$D(\mu)=\inf_{\theta\in(0,1]}\dim_{P,\theta}\mu.$$
\end{remark}
\begin{remark}\label{rem:downward}
Let $0\leq t'\leq t\leq n$ and $\theta\in(0,1]$. Comparing we have 
$$r^{-t'}F_{t',n,\theta}^{\mu}(x,r)\leq r^{(t-t')\theta}\,
r^{-t}F_{t,n,\theta}^{\mu}(x,r).$$ Consequently, 
$$\liminf_{r\to0}r^{-t}F_{t,n,\theta}^{\mu}(x,r)=0,$$
for $\mu\text{-a.e. }x$, for every
$t<\dim_{P,\theta}\mu.$
\end{remark}
\begin{remark}\label{rem:theta_one}
For any $t\in [0,n]$, we have
$F_{t,n,1}^{\mu}(x,r)=F_n^\mu(x,r)$ and therefore
$$\dim_P \mu = \dim_{P,1} \mu.$$
\end{remark}

For simplicity in the notation, for $\theta\in(0,1]$ and $0\leq t\leq s$, we will write
$$\phi_{r,\theta}^{t,s}(z)=
\begin{cases}
1, & \text{if } |z|\leq r, \\[1mm]
r^{t}|z|^{-t}, & \text{if } r< |z| \leq r^{\theta}, \\[1mm]
r^{\theta(s-t)+t}|z|^{-s}, & \text{if } |z| > r^{\theta},
\end{cases}$$
and therefore $F_{t,s,\theta}^{\mu}(x,r)=\int \phi_{r,\theta}^{t,s}(x-y)\,d\mu(y)$.

We now give an equivalent characterization of the packing intermediate
dimension that dispenses with the kernels, $\dim_{P,\theta}\mu$ is the
critical exponent for Frostman-type estimates of $\mu(B(x,r))$ along
windows of scales of the form $[\delta^{1/\theta},\delta]$. 

The characterization is
the packing-dimension analogue of the capacity-theoretic description of
the intermediate dimensions of measures obtained in
\cite[Theorem 5.6]{ANGELINI2026130039}, and the proof follows a similar
strategy. Indeed, the condition on the right-hand side of
\eqref{eq:newdef-formula} involves the same two-scale Frostman-type
condition used in \cite[Definition 3.1]{ANGELINI2026130039} to define
the Minkowski intermediate dimensions of a measure. The difference lies
in the quantifiers: there, the constant and the scales are required to
be uniform over the support (see also
\cite[Lemma 3.3]{ANGELINI2026130039}), whereas in
\eqref{eq:newdef-formula} the scales $\delta(x)$ are allowed to depend
on the point $x$. We also note that the kernels here depend on the
ambient dimension $n$ rather than on the Assouad dimension of $\mu$, so
no regularity hypothesis is required beyond compact support, whereas
the characterization in \cite{ANGELINI2026130039} requires
$\dim_A\mu<\infty$.

\begin{theorem}\label{newdef}Let $\mu$ be a Borel probability measure on $\mathbb{R}^n$ with compact support and let $\theta\in(0,1]$. Then\begin{equation}\label{eq:newdef-formula}    \dim_{P,\theta} \mu = \sup\left\{\begin{array}{c}s\geq 0 \, : \, \forall\, \epsilon>0 \text{ and } \delta_0>0\ \exists\, \delta(x)<\delta_0\ \text{such that} \\[2mm]r^{-s}\mu(B(x,r))<\epsilon\ \ \forall\, r\in[\delta(x)^{1/\theta},\,\delta(x)],\ \ \text{ for } \mu\text{-a.e.\ } x\end{array}\right\}.\end{equation} \end{theorem}

\begin{proof}We will start by proving the $\leq$ inequality.Let $0<s< \dim_{P,\theta} \mu$. Then for every $\epsilon>0$ and $\mu-a.e.\ x$, there exists a sequence $\{\delta_k\}$ with $\delta_k\to 0$ such that

$$\int \phi_{\delta_k,\theta}^{s,n}(x-y)\, d\mu(y) <\epsilon\, \delta_k^{s}.$$

Then for all $r\in [\delta_k,\, \delta_k^\theta]$ we have

$$\mu(B(x,r))\leq \int\mathbbm{1}_{B(0,r)}(x-y)\, d\mu(y)\leq \left(\frac{r}{\delta_k}\right)^s \int \phi_{\delta_k,\theta}^{s,n}(x-y)\, d\mu(y)<\epsilon\, r^s.$$

Hence the $\leq$ inequality follows.

To prove the opposite inequality, let $s>\dim_{P,\theta} \mu$. Then there exists $A\subset \mathbb{R}^n$, of positive $\mu$-measure, such that for all $x\in A$, there exist $\epsilon>0$ and $r_1(x)>0$ such that for all $r\in(0,r_1(x))$,

\begin{equation}\label{equation1}
\epsilon\, r^s < \int \phi_{r,\theta}^{s,n}(x-y)\, d\mu(y).\end{equation}

Moreover, by Lemma~\ref{lem:assouad_measure_growth} and the fact that $\dim_A(\operatorname{supp}\mu)\leq n$, given $\epsilon>0$, for $\mu-a.e.\ x$ there exists $r_2(x)>0$ such that

\begin{equation}\label{equation2}
\mu\bigl(B(x,r^{1/(1+\epsilon)})\bigr)\leq 4^{\,n(1+\epsilon)}\,\frac{r^{n}}{\rho^{\,n(1+\epsilon)}}\, \mu(B(x,\rho)),\end{equation}

for all $\rho\in(0,r_2(x))$ and all $r\in[\rho,1]$. 

Let $x\in A$ and fix $r_0(x)=\min\{r_1(x),r_2(x)\}$. We can suppose that$|\operatorname{supp}(\mu)|=1$.

Let $D=\lceil\log_2 (1/r)\rceil$ and let $M$ be the unique integer satisfying $2^{M-1}r<r^\theta\leq 2^M r$; also let $r_1'$ be sufficiently small to ensure that $2\leq M\leq D-2$ for all $0<r\leq r_1'$, and set $\tilde r_0(x):=\min\{r_0(x),r_1'\}$.Taking $x\in A$, and partitioning the support of $\mu$into consecutive annuli of the form $B(x,2^k r)\setminus B(x,2^{k-1}r)$,$1\leq k\leq D$, we obtain the following, for all $0<r\leq \tilde r_0(x)$:\begin{align*}        \epsilon\, r^s&\leq \int \phi_{r,\theta}^{s,n}(x-y)\,d\mu(y)\\        &\leq \mu(B(x,r))+\sum_{k=1}^{D}\int_{ B(x,2^k r)\setminus B(x,2^{k-1}r) }\phi_{r,\theta}^{s,n}(x-y)\,d\mu(y) \\        &\leq \mu(B(x,r))+\sum_{k=1}^{M}\int_{ B(x,2^k r)\setminus B(x,2^{k-1}r) }2^{-(k-1)s}\,d\mu(y)\\        &\quad + \sum_{k=M+1}^{D}\int_{ B(x,2^k r)\setminus B(x,2^{k-1}r) }r^{\theta(n-s)+s}(2^{k-1}r)^{-n}\,d\mu(y)\\        &\leq \mu(B(x,r))+\sum_{k=1}^{M}\mu(B(x,2^k r))\,2^{-(k-1)s}\\        & \quad +\sum_{k=M+1}^{D} \mu(B(x,2^k r))\,r^{(\theta-1)(n-s)}2^{-n(k-1)}.    \end{align*}

There are $D+1$ terms in the sum above, so at least one of them must be at least the arithmetic mean of the total sum. Therefore, we need to consider three cases:\begin{itemize}    \item $\displaystyle \frac{\epsilon r^s}{D+1}\leq \mu(B(x,r)). $    \item $\displaystyle\frac{\epsilon r^s}{D+1} \leq 2^{s} \mu(B(x,2^{k}r))2^{-ks}=4^s \mu(B(x,2^{k}r)) |B(x,2^{k}r)|^{-s}r^s $for some $k\in\{1,\dots,M\}.$    \item One term of the last sum is greater than or equal to the mean,    i.e., for some $k\in \{M+1,\dots,D\}$,    \begin{align*}\frac{\epsilon r^s}{D+1}&\leq r^{(\theta-1)(n-s)}\mu(B(x,2^{k}r))2^{-(k-1)n}\\&\leq r^{(\theta-1)(n-s)}\mu\bigl(B(x,(2^{k}r)^{1/(1+\epsilon)})\bigr)2^{-(k-1)n}\\&\leq 2^{n(3+2\epsilon)}\, r^{-\theta n \epsilon} r^{-\theta s} r^{s}\,\mu(B(x,r^\theta)),\end{align*}\end{itemize}by using inequality \eqref{equation2}.

Now, since $r^\epsilon\leq \frac{1}{D+1}$ for all sufficiently small $r$,it follows from the three cases above that, for all $x$ in a set of positive $\mu$-measure, there exists $c_0>0$ and $\bar r_0(x)>0$ such that for all $r\in (0,\bar r_0(x))$ there exists $\delta\in [r,r^\theta]$ such that $$c_0 \leq \delta^{-\left(s+\epsilon\left(n+\frac{1}{\theta}\right)\right)}\,\mu(B(x,\delta)).$$Therefore, the supremum on the right-hand side of equality\eqref{eq:newdef-formula} must be less than or equal to $s+\epsilon\left(n+\dfrac{1}{\theta}\right)$, and the result follows by letting $\epsilon\to 0$ and $s\to \dim_{P,\theta}\mu$.
\end{proof}

The next lemma is the key inequality of this section. It bounds the
intermediate dimension from below in terms of the packing dimension of the
measure and the Assouad dimension of its support. Note that for measures with
$\dim_P\mu=\dim_A E$ it shows that the dimension is constant in $\theta$.

\begin{lemma}\label{lem:threshold_lower_bound}
    Let $\mu$ be a Borel probability measure with compact support
    $E:=\operatorname{supp}\mu\subset\mathbb{R}^n$ and let
    $\theta\in(0,1]$. Then
    $$\dim_A E - \frac{\dim_A E - \dim_P \mu}{\theta}\leq \dim_{P,\theta}\mu.$$
\end{lemma}
\begin{proof}
If the left-hand side is
nonpositive there is nothing to prove; so fix $\theta\in(0,1)$ and
assume $\dim_A E-\frac{\dim_A E-\dim_P\mu}{\theta}>0$, which in
particular forces $\dim_P\mu>0$.

Fix $\varepsilon>0$, $a\in(0,1)$ with $a(1+\varepsilon)\dim_A E<n$
(possible since $\dim_A E\leq n$), and $p\in(0,\dim_P\mu)$. Write
$$A:=(1+\varepsilon)\dim_A E,\qquad
s:=A-\frac{A-p}{a\theta},$$
and fix $t$ with $0\leq t<s-A(1-a)$, assuming this range is nonempty.

Let $c>0$. By Lemma~\ref{Lemma 3.9}, for $\mu$-almost every
$x\in\mathbb{R}^n$ there exists a sequence $r_k'\to 0$ such that
\begin{equation}\label{eq:Frostman}
\mu\bigl(B(x,\delta)\bigr)\leq c\,\delta^{s}
\qquad\text{whenever }\delta\in[(r_k')^{a},(r_k')^{a\theta}],
\end{equation}
and by Lemma~\ref{lem:assouad_measure_growth} (applied with the same
parameters $a$ and $\varepsilon$), for $\mu$-almost every $x$ there
exists $\rho_0(x)>0$ such that
\begin{equation}\label{eq:growth}
\mu\bigl(B(x,r)\bigr)\leq\left(\frac{4r}{\rho}\right)^{A}
\mu\bigl(B(x,\rho)\bigr)
\qquad\text{for }0<\rho<\rho_0(x),\ \rho^{a}\leq r\leq 1.
\end{equation}
Fix $x$ in the intersection of both full-measure sets, write
$r_k:=(r_k')^{a}$, so that $[r_k,r_k^{\theta}]=[(r_k')^a,(r_k')^{a\theta}]$
is the window in \eqref{eq:Frostman}, and discard the finitely many $k$
with $r_k^{\theta}\geq\rho_0(x)$. Since each term of the minimum
defining $F_{t,n,\theta}^{\mu}$ dominates it, we may split
$$F_{t,n,\theta}^{\mu}(x,r_k)\leq F_1+F_2+F_3,$$
where $F_1,F_2,F_3$ are the integrals of $1$, of $r_k^t|x-y|^{-t}$ and
of $r_k^{\theta(n-t)+t}|x-y|^{-n}$ over the regions $|x-y|\leq r_k$,
$r_k<|x-y|\leq r_k^{\theta}$ and $|x-y|>r_k^{\theta}$, respectively.

\emph{Estimate of $F_1$.} By \eqref{eq:Frostman} with $\delta=r_k$, and
since $t<s$,
$$F_1=\mu\bigl(B(x,r_k)\bigr)\leq c\,r_k^{s}\leq c\,r_k^{t}.$$

\emph{Estimate of $F_2$.} By a standard computation we have,
for $0<\alpha<\beta$,
$$\int_{\alpha<|x-y|\leq\beta}|x-y|^{-t}\,d\mu(y)
\leq \beta^{-t}\mu\bigl(B(x,\beta)\bigr)
+t\int_{\alpha}^{\beta}u^{-t-1}\mu\bigl(B(x,u)\bigr)\,du.$$
Applying this with $\alpha=r_k$, $\beta=r_k^{\theta}$ and using
\eqref{eq:Frostman} for $u\in[r_k,r_k^{\theta}]$,
$$F_2\leq r_k^{t}\left(c\,r_k^{\theta(s-t)}
+ct\int_{r_k}^{r_k^{\theta}}u^{s-t-1}\,du\right)
\leq c\left(1+\frac{t}{s-t}\right)r_k^{\theta(s-t)}\,r_k^{t}.$$

\emph{Estimate of $F_3$.} We claim that, for $k$ large enough,
\begin{equation}\label{eq:tail_bound}
\mu\bigl(B(x,u)\bigr)\leq 4^{A}c\,u^{aA}\,r_k^{-\theta(A-s)}
\qquad\text{for all }u>r_k^{\theta}.
\end{equation}
Indeed, for $u\in(r_k^{\theta},1]$ we have $u\leq u^{a}\leq 1$ and
$u^{a}\geq(r_k^{\theta})^{a}$, so \eqref{eq:growth} with $\rho=r_k^{\theta}$
and radius $u^{a}$, followed by \eqref{eq:Frostman} with
$\delta=r_k^{\theta}$, gives
$$\mu\bigl(B(x,u)\bigr)\leq\mu\bigl(B(x,u^{a})\bigr)
\leq\left(\frac{4u^{a}}{r_k^{\theta}}\right)^{A}
\mu\bigl(B(x,r_k^{\theta})\bigr)
\leq 4^{A}c\,u^{aA}\,r_k^{\theta(s-A)};$$
and for $u>1$ we simply use $\mu(B(x,u))\leq 1\leq 4^{A}c\,u^{aA}\,
r_k^{-\theta(A-s)}$, valid for $k$ large since $A>s$ and $r_k\to 0$.
Then we have,
\begin{align*}
    F_3&\leq r_k^{\theta(n-t)+t} n\int_{r_k^{\theta}}^{\infty}
u^{-n-1}\mu\bigl(B(x,u)\bigr)\,du\\
&\leq \frac{4^{A}c\,n}{n-aA}\,r_k^{\theta(n-t)+t}\,
r_k^{-\theta(A-s)}\,(r_k^{\theta})^{aA-n}\\
&\leq \frac{4^{A}c\,n}{n-aA}\,r_k^{t},
\end{align*}
using \eqref{eq:tail_bound}, $aA<n$ and $t<s-A(1-a)$.

Combining the three estimates,
$F_{t,n,\theta}^{\mu}(x,r_k)\leq K\,c\,r_k^{t}$ for all large $k$, where
$K$ depends on $a,\varepsilon,\theta,t,s,n$ but not on $c$ or $k$. Hence
$$\liminf_{r\to 0}r^{-t}F_{t,n,\theta}^{\mu}(x,r)\leq K\,c.$$
Taking a sequence $c_j\to 0$ and intersecting the corresponding
full-measure sets of $x$, we conclude that
$\liminf_{r\to 0}r^{-t}F_{t,n,\theta}^{\mu}(x,r)=0$ for $\mu$-a.e.\ $x$
and every $t<s-A(1-a)$, and therefore
\begin{align*}
 a(1+\varepsilon)\dim_A E-\frac{(1+\varepsilon)\dim_A E-p}{a\theta}
\leq \dim_{P,\theta}\mu\end{align*}

Finally, letting $\varepsilon\to 0$, $a\to 1$ and $p\to\dim_P\mu$, the left-hand side
converges to $\dim_A E-\frac{\dim_A E-\dim_P\mu}{\theta}$, which
completes the proof.
\end{proof}

\section{Full packing dimension threshold}

We are now ready to prove the main result of this paper: the threshold
parameter $D(\mu)$ characterizes exactly when the orthogonal projections
of $\mu$ onto almost every $m$-dimensional subspace attain full packing
dimension.

\begin{theorem}\label{maintheorem}
    Let $\mu$ be a Borel probability measure with compact support on
    $\mathbb{R}^n$ and let $m\in\mathbb{N}$ with $1\leq m\leq n$. Then
    $$\dim_P \mu_V = m \text{ for }\gamma_{n,m}-a.e.\ V\in G(n,m) \quad\Longleftrightarrow\quad m\leq D(\mu).$$
\end{theorem}

The proof is split into the following two lemmas.

\begin{lemma}\label{mleqdu}
    Let $\mu$ be a Borel probability measure with compact support
    $E\subset\mathbb{R}^n$ and let $m\in\mathbb{N}$, $1\leq m\leq n$.
    
    If
    $\dim_P \mu_V = m \text{ for }\gamma_{n,m}-a.e.\ V\in G(n,m)$, then $m\leq D(\mu)$.
\end{lemma}
\begin{proof}

Let $V\in G(n,m)$. For $0\leq t\leq m$ we have
$\phi_{r,\theta}^{t,n}(z)\leq\phi_{r,\theta}^{t,m}(z)$, since the two
kernels differ only in their last term and
$r^{\theta(n-t)+t}|z|^{-n}\leq r^{\theta(m-t)+t}|z|^{-m}$ whenever
$|z|>r^{\theta}$; moreover $\phi_{r,\theta}^{t,m}$ is a nonincreasing
function of $|z|$ and $|\pi_V(z)|\leq|z|$. Therefore
\begin{align*}
    F_{t,n,\theta}^{\mu}(x,r)&=\int \phi_{r,\theta}^{t,n}(x-y)\, d\mu(y)
    \leq\int \phi_{r,\theta}^{t,m}\bigl(\pi_V(x)-\pi_V(y)\bigr)\, d\mu(y)\\
    &= \int \phi_{r,\theta}^{t,m}\bigl(\pi_V(x)-w\bigr)\, d\mu_V(w)
    =F_{t,m,\theta}^{\mu_V}\bigl(\pi_V(x),r\bigr),
\end{align*}
the last equality because $\mu_V$ is a measure on $V\cong\mathbb{R}^m$.
It follows that
\begin{equation}\label{eq:proj_chain}
\dim_{P,\theta}\mu_V \leq \dim_{P,\theta}\mu
\qquad\text{for all }V\in G(n,m),\ \theta\in(0,1].
\end{equation}

Now assume $\dim_P\mu_V=m$ for $\gamma_{n,m}$-almost
every $V\in G(n,m)$. Fix such a $V$, by monotony we have
$$m=\dim_P\mu_V\leq\dim_A(\operatorname{supp}\mu_V)\leq m.$$
Applying
Lemma~\ref{lem:threshold_lower_bound} to $\mu_V$,
$$\dim_{P,\theta}\mu_V\geq
\dim_A(\operatorname{supp}\mu_V)-\frac{\dim_A(\operatorname{supp}\mu_V)
-\dim_P\mu_V}{\theta}=m$$
for every $\theta\in(0,1].$
Combining with \eqref{eq:proj_chain}, $m\leq\dim_{P,\theta}\mu$ for
every $\theta\in(0,1]$, and by Remark~\ref{rem:monotone_theta} we
conclude $m\leq D(\mu)$.
\end{proof}

\begin{lemma}\label{dumleq}
    Let $\mu$ be a Borel probability measure with compact support on
    $\mathbb{R}^n$ and let $m\in\mathbb{N}$ with $1\leq m\leq n$.
    
    If
    $m\leq D(\mu)$, then $$\dim_P \mu_V = m \text{ for }\gamma_{n,m}-a.e.\, V\in G(n,m).$$
\end{lemma}
\begin{proof}
Since $\dim_P\mu_V\leq m$ for all $V$, we only need the reverse
inequality.

First note that
\begin{equation}\label{eq:comparison_m_n}
    F_m^{\mu}(x,r)\leq F_{t,m,\theta}^\mu(x,r)
    \leq F_{t,n,\theta}^\mu(x,r) + \mu(\mathbb{R}^n)\, r^{t(1-\theta)},
\end{equation}
valid for all $x\in\mathbb{R}^n$, $r\in(0,1)$, $\theta\in(0,1]$ and
$t\in[0,m]$. 

Let $t<m$ and $\theta\in(0,1]$. By Remark~\ref{rem:monotone_theta},
$$t<m\leq D(\mu)=\inf_{\theta'\in(0,1]}\dim_{P,\theta'}\mu
\leq \dim_{P,\theta}\mu,$$
therefore
$$\liminf_{r\to 0}r^{-t}F_{t,n,\theta}^{\mu}(x,r)=0
\qquad\text{for }\mu\text{-a.e.\ }x\in\mathbb{R}^{n}.$$
Hence, for $\mu$-a.e.\ $x$ there exists a sequence $r_k\to 0$ with
$F_{t,n,\theta}^\mu(x,r_k)\leq r_k^{t}$, and \eqref{eq:comparison_m_n}
gives
$$F_m^{\mu}(x,r_k)\leq 2\,r_k^{t(1-\theta)}.$$
Now observe that $F_m^\mu$ does not depend on $t$; therefore, for every
$t'<t(1-\theta)$,
$$r_k^{-t'}F_m^{\mu}(x,r_k)\leq 2\,r_k^{\,t(1-\theta)-t'}
\longrightarrow 0\qquad (k\to\infty),$$
so that $\liminf_{r\to 0}r^{-t'}F_m^\mu(x,r)=0$ for $\mu$-a.e.\ $x$, and
consequently $t'\leq\dim_P\mu_V$ for $\gamma_{n,m}-a.e.\ V\in G(n,m)$ by Theorem \ref{ProjPacking}. Letting $t'\to t(1-\theta)$, $\theta\to 0$ and $t\to m$ (along countable sequences,
intersecting the corresponding full-measure sets) we obtain

$$m\leq\dim_P\mu_V$$
for $\gamma_{n,m}-a.e.\, V\in G(n,m)$.
\end{proof}

\begin{proof}[Proof of Theorem~\ref{maintheorem}]
The proof follows by Lemma~\ref{mleqdu} and Lemma~\ref{dumleq}.
\end{proof}

\section{Relations with the Assouad dimension of the support}

The following theorem quantifies how much the $m$-dimensional profile of
a measure can drop below its packing dimension. Note that the
measure appears in the bound only through $\dim_P\mu$; the size of the
possible loss is governed by the support alone, and is therefore common
to all measures supported on a given set.

\begin{theorem}\label{thm:profile_assouad_bound}
    Let $\mu$ be a Borel probability measure with compact support
    $E\subset\mathbb{R}^{n}$ and let $0<m\leq n$. Then
    $$\dim_P^m \mu \geq \dim_P \mu - \max\{0,\ \dim_A E - m\}.$$
\end{theorem}

\begin{proof}
If $\dim_P\mu=0$ the right-hand side is nonpositive and there is nothing
to prove, so assume $\dim_P\mu>0$; write $A:=\dim_A E$ and note that
$0<\dim_P\mu\leq A$.

Let $0<p<\dim_P\mu$ and let $0\leq t<\min\{p,\ p+m-A\}$. Choose
$\varepsilon>0$ and $a\in(0,1)$ such that $m\neq A(1+\varepsilon)$ and
\begin{equation}\label{eq:choice_params}
t<\min\bigl\{p-A(1+\varepsilon)(1-a),\ p+m-A(1+\varepsilon)\bigr\}.
\end{equation}

Then, for $\mu$-almost every $x$ there
exists a sequence $r_k\to 0$ with
\begin{equation}\label{eq:frostman_seq}
\mu\bigl(B(x,r_k)\bigr)\leq r_k^{\,p},
\end{equation}
and by Lemma~\ref{lem:assouad_measure_growth}, applied with the
parameters $a$ and $\varepsilon$, for $\mu$-almost every $x$ there exists
$\rho_0(x)>0$ such that
\begin{equation}\label{eq:growth_thm}
\mu\bigl(B(x,r)\bigr)\leq\left(\frac{4r}{\rho}\right)^{A(1+\varepsilon)}
\mu\bigl(B(x,\rho)\bigr),
\end{equation}
for $0<\rho<\rho_0(x),\
\rho^{a}\leq r\leq 1.$
Fix $x$ in the intersection of both full-measure sets and discard the
finitely many $k$ with $r_k\geq\rho_0(x)$. Write $B_x(h):=\mu(B(x,h))$. Since
$\min\{1,r^m|z|^{-m}\}=m\,r^m\int_{\max\{r,|z|\}}^{\infty}h^{-m-1}dh$,
we have
$$F_{m}^{\mu}(x,r)=m\,r^{m}\int_{r}^{\infty}h^{-m-1}B_{x}(h)\,dh.$$
Accordingly, we write $F_{m}^{\mu}(x,r_k)=I_1+I_2+I_3$, where
\begin{align*}
I_1&:=m\,r_k^{m}\int_{r_k}^{r_k^{a}}h^{-m-1}B_{x}(h)\,dh,\\
I_2&:=m\,r_k^{m}\int_{r_k^{a}}^{1}h^{-m-1}B_{x}(h)\,dh,\\
I_3&:=m\,r_k^{m}\int_{1}^{\infty}h^{-m-1}B_{x}(h)\,dh.
\end{align*}

Since 
$\int_{r_k}^{r_k^{a}}h^{-m-1}dh\leq r_k^{-m}/m$, using
\eqref{eq:growth_thm} with $\rho=r_k$, $r=r_k^{a}$ and then
\eqref{eq:frostman_seq},
$$I_1\leq B_{x}(r_k^{a})
\leq\left(\frac{4r_k^{a}}{r_k}\right)^{A(1+\varepsilon)}B_{x}(r_k)
\leq 4^{A(1+\varepsilon)}\,r_k^{\,p-A(1+\varepsilon)(1-a)}.$$
Similarly, using \eqref{eq:growth_thm} with $\rho=r_k$ and $r=h$ for
$h\in[r_k^{a},1]$, together with \eqref{eq:frostman_seq},
$$I_2\leq m\,4^{A(1+\varepsilon)}\,r_k^{\,p+m-A(1+\varepsilon)}
\int_{r_k^{a}}^{1}h^{A(1+\varepsilon)-m-1}\,dh,$$
and finally $I_3\leq m\,r_k^{m}\int_{1}^{\infty}h^{-m-1}dh=r_k^{m}$.

We estimate the remaining integral in the two possible cases. If
$m<A(1+\varepsilon)$, then
$$\int_{r_k^{a}}^{1}h^{A(1+\varepsilon)-m-1}\,dh
\leq\frac{1}{A(1+\varepsilon)-m},$$
so that $I_2\leq\frac{m\,4^{A(1+\varepsilon)}}{A(1+\varepsilon)-m}\,
r_k^{\,p+m-A(1+\varepsilon)}$, and the exponent exceeds $t$ by
\eqref{eq:choice_params}. If $m>A(1+\varepsilon)$, then
$$\int_{r_k^{a}}^{1}h^{A(1+\varepsilon)-m-1}\,dh
\leq\frac{r_k^{-a(m-A(1+\varepsilon))}}{m-A(1+\varepsilon)},$$
so that $I_2\leq\frac{m\,4^{A(1+\varepsilon)}}{m-A(1+\varepsilon)}\,
r_k^{\,p+(1-a)(m-A(1+\varepsilon))}$, and the exponent exceeds $p>t$.
The exponent of $I_1$ exceeds $t$ by \eqref{eq:choice_params}, and so
does that of $I_3$, since $p\leq\dim_P\mu\leq A\leq A(1+\varepsilon)$
gives $m\geq p+m-A(1+\varepsilon)>t$.

Consequently $r_k^{-t}F_{m}^{\mu}(x,r_k)\to 0$, so that
$$\liminf_{r\to 0}r^{-t}F_{m}^{\mu}(x,r)=0
\qquad\text{for }\mu\text{-a.e.\ }x\in\mathbb{R}^n,$$
and therefore $t\leq\dim_P^m\mu$. As $t<\min\{p,\ p+m-A\}$ and
$p<\dim_P\mu$ were arbitrary,
$$\dim_P^m\mu\geq\min\{\dim_P\mu,\ \dim_P\mu+m-A\}
=\dim_P\mu-\max\{0,\ \dim_A E-m\}.$$
\end{proof}

\begin{remark}\label{rem:not_adaptation}
As mentioned in the introduction, Theorem~\ref{thm:profile_assouad_bound}
is not an adaptation of the corresponding result for sets, obtained by
Falconer, Fraser and Shmerkin \cite[Theorem 2.2]{falconer2021assouad}. Their argument
relies on the capacity formulation of the box dimension profiles, in
which the profile of a set is defined through an infimum over all
measures supported on it; this allows one to choose, at each scale $r$, a
measure adapted to that scale, and their proof takes it to be the
normalized counting measure on a maximal $r$-separated subset of $F$,
the packing case being deduced afterwards by a Baire category argument.
In the present setting no such freedom is available: the measure $\mu$ is
given and is the same at every scale, and the profile is defined through
the behaviour of its own potentials at $\mu$-almost every point.

On the contrary, Theorem~\ref{thm:profile_assouad_bound} yields the Assouad dimension version of the set-theoretic statement as a corollary. Indeed, let
$E\subset\mathbb{R}^n$ be an Borel set. For every
$\mu\in\mathcal{M}_c^+(E)$ we have $\dim_A(\operatorname{supp}\mu)\leq
\dim_A E$, so that
$$\dim_P^m\mu\geq\dim_P\mu-\max\{0,\ \dim_A E-m\},$$
and taking the supremum over $\mu\in\mathcal{M}_c^+(E)$, together with the equality
$\dim_P^m E=\sup\{\dim_P^m\mu\ : \mu\in\mathcal{M}_c^+(E)\}$ gives
$$\dim_P^m E\geq\dim_P E-\max\{0,\ \dim_A E-m\},$$
which is \cite[Theorem 2.2]{falconer2021assouad} in its Assouad dimension form.
\end{remark}

\section*{Funding}
The research was partially supported by grants PICT 2022-4875 (ANPCyT), PIP 202287/22 (CONICET), and PROICO 3-0720 ``An\'alisis Real y Funcional. Ec. Diferenciales''.

\end{document}